\documentclass[12pt]{article}
\usepackage{booktabs}
\usepackage{caption}
\usepackage{mathrsfs}
\usepackage{amsmath}
\usepackage{amsfonts,amsthm,amssymb,mathrsfs,bbding}
\usepackage{graphics,multicol}
\usepackage{graphicx}
\usepackage{color}
\usepackage{enumerate}
\usepackage{caption}
\usepackage{rotating}
\usepackage{lscape}
\usepackage{longtable}
\allowdisplaybreaks[4]
\usepackage{tabularx}
\usepackage[colorlinks=true,anchorcolor=blue,filecolor=blue,linkcolor=blue,urlcolor=blue,citecolor=blue]{hyperref}
\usepackage{extarrows}
\usepackage{cite}
\usepackage{tikz}
\usepackage{latexsym,bm}
\usepackage{mathtools}
\evensidemargin=\oddsidemargin\topmargin=-1.5cm

\newtheorem{thm}{Theorem}

\newtheorem{ques}{Question}
\newtheorem{lem}[thm]{Lemma}

\newtheorem{claim}{Claim}
\theoremstyle{definition}

\addtocounter{section}{0}
\begin{document}

	\title{\bf Spectral condition for the existence of a chorded cycle}
	\author{{Jiaxin Zheng,  Xueyi Huang\footnote{Corresponding author.}\setcounter{footnote}{-1}\footnote{\emph{E-mail address:} huangxymath@163.com}},  Junjie Wang\\[2mm]
	\small School of Mathematics, East China University of Science and Technology,\\
	\small Shanghai 200237, China}

	\date{}
	\maketitle
	{\flushleft\large\bf Abstract }  A chord
	of a cycle $C$ is an edge joining two non-consecutive vertices of $C$. A cycle $C$ in a graph $G$ is chorded if the vertex set of $C$ induces at least one chord. In this paper, we prove that if  $G$ is a graph with order $n\geq 6$ and $\rho(G)\geq \rho(K_{2,n-2})$, then $G$ contains a chorded cycle unless $G\cong K_{2,n-2}$. This gives one answer to a question posed by  Gould [Results and problems on chorded cycles: A survey,  Graphs Combin. 38 (2022)  189].
	
	\begin{flushleft}
	\end{flushleft}

	\section{Introduction}
	All graphs considered in this paper are simple and undirected. Let $G$ be a graph with vertex set $V(G)$ and  edge set $E(G)$. For a graph $G$, we refer to  the number of vertices of $G$ as the \textit{order} of $G$ and denote it by $|G|$. Let $e(G):=|E(G)|$ denote the number of edges in $G$.  For any $v\in V(G)$, let $d_G(v)$ denote the degree of $v$ in $G$, and $N_G(v)$ denote the set of vertices adjacent to $v$ in $G$.  For any vertex subset $S\subseteq V(G)$, we denote by $G[S]$ the subgraph of $G$ induced by $S$, and $e(S):=e(G[S])$. For any two disjoint subsets $S$ and $T$ of $V(G)$, we denote by $E(S,T):=E_G(S,T)$ the set of edges of $G$ between $S$ and $T$, and let $e(S,T):=|E(S,T)|$. The \textit{join} of two graphs $G_1$ and $G_2$, denoted by $G_1\nabla G_2$,  is the graph obtained from the vertex-disjoint union $G_1\cup G_2$ by adding all possible edges between $G_1$ and $G_2$. 
	
	The \textit{adjacency matrix} of $G$ is defined as $A(G)=(a_{u,v})_{u,v\in V(G)}$, where $a_{u,v}=1$ if $u$ and $v$ are adjacent in $G$, and $a_{u,v}=0$ otherwise.  The largest eigenvalue of  $A(G)$  is called the \textit{spectral radius} of $G$, and denoted by $\rho(G)$. In recent years, the problem of finding spectral conditions for graphs having certain structural properties or containing specified kinds of subgraphs has received considerable attention. Cioab\u{a}, Gregory and Haemers \cite{CGH} found a best upper bound on the third largest eigenvalue that is sufficient to guarantee that an $n$-vertex $k$-regular graph $G$ has a perfect matching when $n$ is even, and a matching of order $n-1$ when $n$ is odd. Fiedler and Nikiforov \cite{FN} gave a spectral radius condition for graphs to have a Hamilton cycle or Hamilton path. Li and Ning \cite{LN} provided a tight spectral radius condition for graphs with bounded minimum degree to have a Hamilton cycle or Hamilton path. Cioab\u{a}, Feng, Tait and Zhang \cite{CFTZ} provided a tight spectral radius for graphs to contain a friendship graph of given order as a subgraph. For more results on this topic, we refer the reader to  \cite{BZ,O16,O22,OC,OPPZ,WKX,ZL}, and references therein.

	 A \textit{chord} of a cycle $C$ is an edge joining two non-consecutive vertices of $C$. A cycle $C$ in a graph $G$ is \textit{chorded} if the vertex set of $C$ induces at least one chord. In 1961, P\'{o}sa \cite{P}  asked a very natural question:

\begin{ques}\label{ques0}	 
What conditions imply a graph contains a chorded cycle?
\end{ques} 

In \cite{P}, P\'{o}sa  provided one answer by showing that every graph of order $n$ with at least $2n-3$ edges contains a chorded cycle. From then on, P\'{o}sa's question had aroused a lot of interest. In 2022,  Gould \cite{G} surveyed results and problems that relate to Posa's question on chorded cycles in graphs, which include  sufficient conditions for the existence of a
chorded cycle, or sets of chorded cycles, or cycles with multiple chords, or chorded cycles with additional properties. At the end of the survey paper \cite{G}, Gould asked the following quesiton.
	 
	\begin{ques}\label{ques1}
	    What spectral conditions imply the existence of a chorded cycle in a graph?
	\end{ques}
    
In this paper, we provide one answer for Gould's question by using the spectral radius of graphs.

	\begin{thm}\label{thm::1}
		Let $G$ be a graph of order $n \geq 6$. If $\rho(G) \geq \rho(K_{2,n-2})$, then $G$ contains a chorded cycle unless $G\cong K_{2,n-2}$.
	\end{thm}

	\section{Preliminaries}
	
	In this section, we introduce some notions and lemmas, which are useful in the proof of Theorem \ref{thm::1}. The first two results are well-known, and one can find them in \cite{XLLS} and \cite{B}, respectively.

	\begin{lem}(\cite{XLLS})\label{lem::1}
		Let $G$ be a connected graph. For $u,v\in V(G)$, suppose $N\subseteq N(v)\backslash (N(u)\cup \{u\})$. Let $G'=G-\{vw: w\in N\}+\{uw: w\in N\}$. If $N\neq \emptyset$ and $\boldsymbol{x}=(x_v)_{v\in V(G)}$ is the Perron vector of $G$ such that $x_{u}\geq x_{v}$, then $\rho(G')>\rho(G)$.
	\end{lem}

	\begin{lem}(\cite{B})\label{lem::2}
		Let $G$ be a connected graph, and let $H$ be a  proper subgraph of $G$. Then $\rho(H)<\rho(G)$.
	\end{lem}

		\begin{lem}(\cite{ZLG15})\label{lem::7}
		Let $G =(X,Y)$ be a bipartite graph, where $|X|\geq r$ and $|Y|\geq r-1\geq1$.
        If $G$ does not contain a copy of $P_{2r+1}$ with both endpoints in $X$, then
        $$e(G)\leq (r-1)|X| + r|Y |-r(r-1).$$
         Equality holds if and only if $G\cong K_{|X|,|Y|}$, where $|X| =r$ or $|Y| =r-1$.
	    \end{lem}

    The \textit{friendship graph} $F_k$ ($k\geq 1$) is the graph of order $2k+1$ consisting of $k$ edge-disjoint triangles that meet in a single vertex.

    \begin{lem}(\cite{V})\label{lem::3}
		The eigenvalues of $F_k$ are 
		$(1\pm \sqrt{1+8k})/2$ and  $\pm 1$ (with multiplicity $k-1$). 
	\end{lem}

        Let $M$ be a real $n$ $\times$ $n$ matrix, and let $\Pi=\{X_{1},X_{2}, \ldots,X_{k}\}$ be a partition of $[n]=\{1,2,\ldots,n\}$. Then the matrix $M$ can be correspondingly partitioned as
        $$
	M=\left(\begin{array}{ccccccc}
		M_{1,1}&M_{1,2}&\cdots&M_{1,k}\\
            M_{2,1}&M_{2,2}&\cdots&M_{2,k}\\
            \vdots&\vdots&\ddots&\vdots\\
            M_{k,1}&M_{k,2}&\cdots&M_{k,k}\\
	\end{array}\right).
	$$
 The \textit{quotient matrix} of $M$ with respect to $\Pi$ is  the matrix $B_{\Pi}=(b_{i,j})^{k}_{i,j=1}$ with
        $$
            b_{i,j}=\frac{1}{|X_{i}|}\boldsymbol{j}^{T}_{|X_{i}|}M_{i,j}\boldsymbol{j}_{|X_{j}|}
        $$
for all $i,j \in \{ 1,2,\ldots,k \}$, where $\boldsymbol{j}_{s}$ denotes the all ones vector in $\mathbb{R}^{s}$. If each block $M_{i,j}$ of $M$ has constant row sum $b_{i,j}$, then $\Pi$ is called an \textit{equitable partition}, and the quotient matrix $B_\Pi$ is called an  \textit{equitable quotient matrix} of $M$. Also, if the eigenvalues of $M$ are real, we denote them by $\lambda_1(M)\geq \lambda_2(M)\geq \cdots \geq \lambda_n(M)$. 

        \begin{lem}(Brouwer and Haemers \cite[p. 30]{BH}; Godsil and Royle \cite[pp.196--198]{GR})\label{lem::4}
            Let $M$ be a real symmetric matrix, and let $B$ be an equitable quotient matrix of $M$. Then the eigenvalues of $B$ are also eigenvalues of $M$. Furthermore, if $M$ is nonnegative and irreducible, then
            $$ \lambda_1(M)=\lambda_1(B).
            $$
	\end{lem}

For any vertex $u$ of $G$, let $G^{u}$ denote the graph obtained from $G$ by attaching a pendant vertex at $u$.

    \begin{lem}\label{lem::5}
     Let $n\geq 6$ be an even integer. If $u$ is  the central vertex of  $F_{{\frac{n-2}{2}}}$, then
     $$\rho(F^{u}_{\frac{n-2}{2}})< \sqrt{2n-4}.$$
    \end{lem}
    \begin{proof}
Suppose that $v$ is the pendant vertex attaching at $u$ in  $F^{u}_{\frac{n-2}{2}}$. Let $V_{1}=\{u\}$, $V_2=V(F_{{\frac{n-2}{2}}})\backslash \{u,v\}$ and $V_3=\{v\}$. Then it is easy to see that the partition  $\Pi: V(F^{u}_{\frac{n-2}{2}})=V_1\cup V_2\cup V_3$ is an equitable partition of $F^{u}_{\frac{n-2}{2}}$, and the corresponding quotient matrix is 	
     $$
	  B_\Pi=\left(\begin{array}{ccccccc}
		0&n-2&1\\
		1&1&0\\
		1&0&0\\
 	   \end{array}\right).
	$$
 Let $f(x)$ be the characteristic polynomial of $B_\Pi$. Then 
    \begin{equation*}
			f(\sqrt{2n-4})	=(n-3)\sqrt{2n-4}-(2n-5)=\frac{(n-6)(2n^2-8n+14)+23}{(n-3)\sqrt{2n-4}+2n-5} > 0
	\end{equation*}
	as $n\geq 6$. We claim that $\lambda_1(B_\Pi) <  \sqrt{2n-4}$. As $f(2)=7-2n<0$ and $2<\sqrt{2n-4}$, we have $\lambda_2(B_\Pi) <2$ or $\lambda_3(B_\Pi) > 2$. If $\lambda_2(B_\Pi) <2$, then $\lambda_1(B_\Pi) <  \sqrt{2n-4}$, as desired. If $\lambda_3(B_\Pi) > 2$, then $\lambda_1(B_\Pi)+\lambda_2(B_\Pi)+\lambda_3(B_\Pi)> 6$. On the other hand,  $\lambda_1(B_\Pi)+\lambda_2(B_\Pi)+\lambda_3(B_\Pi)=\mathrm{trace}(B_\Pi)=1$,  a contradiction. Therefore, by Lemma \ref{lem::4},
	$$
	\rho(F^{u}_{\frac{n-2}{2}})=\lambda_1(B_\Pi)< \sqrt{2n-4},
	$$
	and our results follows.
    \end{proof}

    Let $a\geq 2$ and $n\geq \max\{6,a+4\}$ be two integers with the same parity. We define  $K_{2,a}\bullet F_{\frac{n-a-2}{2}}$ (resp. $K_{2,a} * F_{\frac{n-a-2}{2}}$) as the graph of order $n$ obtained by identifying a vertex $u$ of $K_{2,a}$ belonging to the part of size $a$ (resp. $2$) with the central vertex of $F_{\frac{n-a-2}{2}}$.

    \begin{lem}\label{lem::6}
    Let $a\geq 2$ and $n\geq \max\{6,a+4\}$ be two integers with the same parity. Then
     $$\max\{\rho(K_{2,a}\bullet F_{\frac{n-a-2}{2}}),\rho(K_{2,a}* F_{\frac{n-a-2}{2}})\}< \sqrt{2n-4}.$$
    \end{lem}
    
    \begin{proof}
Suppose that $A$ and $B$ are the two parts of $K_{2,a}$ with size $a$ and $2$, respectively. Let $u$ be the central vertex of $F_{\frac{n-a-2}{2}}$. Set $V_{1}=\{u\}$, $V_2=V(K_{2,a}\bullet F_{\frac{n-a-2}{2}})\backslash V(K_{2,a})$, $V_3=B$ and $V_4=A\backslash \{u\}$. Then it is easy to see that the partition  $\Pi: V(K_{2,a}\bullet F_{\frac{n-a-2}{2}})=V_1\cup V_2\cup V_3\cup V_4$ is an equitable partition of $K_{2,a}\bullet F_{\frac{n-a-2}{2}}$, and the corresponding quotient matrix is 	
    $$
	B_\Pi=\left(\begin{array}{ccccccc}
		0&n-a-2&2&0\\
		1&1&0&0\\
		1&0&0&a-1\\
		0&0&2&0\\
	\end{array}\right).
	$$
	Let $f(x)$ denote the characteristic polynomial of $B_\Pi$. Then
    \begin{equation*}
	    \begin{aligned}
			f(\sqrt{2n-4})&=2n^2-10n-(2n-4)\sqrt{2n-4}-2a^2+2a(\sqrt{2n-4}+1)+12\\
			&\ge 2n^2-10n-(2n-4)\sqrt{2n-4}-2(n-4)^2\\
			&~~~+2(n-4)(\sqrt{2n-4}+1)+12~~(\mbox{as}~2\leq a\leq n-4)\\
			&= 4(2n-7-\sqrt{2n-4})\\
			&> 0.
		\end{aligned}
	\end{equation*}
We claim that $\lambda_1(B_\Pi) <  \sqrt{2n-4}$. Since $f(\sqrt{2a})=2a-2n+4<0$ and $\sqrt{2a} < \sqrt{2n-4}$, we have $\lambda_2(B_\Pi) <\sqrt{2a}$ or $\lambda_3(B_\Pi) >\sqrt{2a}$. If $\lambda_2(B_\Pi) <\sqrt{2a}$, then $\lambda_1(B_\Pi) <  \sqrt{2n-4}$, as desired. If $\lambda_3(B_\Pi) > \sqrt{2a}$, then $\lambda_2(B_\Pi)+\lambda_3(B_\Pi)> 2\sqrt{2a} > 1$. On the other hand,  since $\lambda_1(B_\Pi)+\lambda_2(B_\Pi)+\lambda_3(B_\Pi)+\lambda_4(B_\Pi)=\mathrm{trace}(B_\Pi)=1$ and $\lambda_1(B_\Pi)+\lambda_4(B_\Pi)\geq 0$, we obtain $\lambda_2(B_\Pi)+\lambda_3(B_\Pi)\leq 1$, which is  a contradiction. Therefore, by Lemma \ref{lem::4},
	$$
	\rho(K_{2,a}\bullet F_{\frac{n-a-2}{2}})=\lambda_1(B_\Pi)< \sqrt{2n-4}.
	$$
Similarly,  we can prove that $\rho(K_{2,a} * F_{\frac{n-a-2}{2}})< \sqrt{2n-4}$.
    \end{proof}
    

	\section{Proof of Theorem \ref{thm::1}}
	
In this section, we shall give the proof of  Theorem \ref{thm::1}.

	{\flushleft \it Proof of Theorem \ref{thm::1}.} Suppose that $G$ has the maximum spectral radius among all graphs without a chorded cycle. First of all, we claim that $G$ is connected. If not, then we can obtain a new graph $G'$ by adding a new edge between the component having $\rho(G)$ as an eigenvalue and any other component in $G$. Clearly, $G'$ does not contain a chorded cycle. By Lemma \ref{lem::2}, we have $\rho(G') > \rho(G)$, contrary to the assumption. 
	
Now suppose that  $\boldsymbol{x}=(x_{v_1},x_{v_2},\ldots,x_{v_n})^T$ is the Perron vector of $G$, and that $u^*$ is a vertex of $G$ such that $x_{u^*}=\max\{x_v: v\in V(G)\}$. Let $A=N_G(u^*)$, $B=V(G)\backslash (A\cup\{u^*\})$ and $\gamma(u^*)=|A|+2e(A)+e(A,B)$. Since $K_{2,n-2}$ does not contain a chorded cycle, we have
	\begin{equation}\label{eq1}
    \rho(G)\geq \rho(K_{2,n-2})=\sqrt{2n-4}.
    \end{equation} 
On the other hand, by using the eigenvalue-eigenvector equation, we obtain
    \begin{equation*}
        \begin{aligned}
            \rho^2(G)x_{u^*}&=\sum_{v\sim u^*}\sum_{w\sim v}x_w\\
    &=|A|x_{u^*}+\sum_{v\in A}d_A(v)x_v+\sum_{w\in B}d_A(w)x_w\\
    &\leq\left(|A|+2e(A)+e(A,B)\right) x_{u^*}\\
    &=\gamma(u^*)x_{u^*}.
        \end{aligned}
    \end{equation*}
Combining this with \eqref{eq1} yields that
   \begin{equation}\label{eq2}
   \gamma(u^*)\geq 2n-4.
   \end{equation} 
   We consider the following two situations.
   
   {\flushleft {\it Case 1.} $e(A)=0$.}

In this situation, we claim that $|B|\geq 1$, since otherwise $G\cong K_{1,n-1}$, and  $\rho(G)=\sqrt{n-1}<\sqrt{2n-4}$, a contradiction. Furthermore, 
since  $\rho(G)x_{u^*}=\sum_{v\in A}x_v\leq |A|x_{u^*}$, by (\ref{eq1}), we obtain $|A|\geq \rho(G)\geq \sqrt{2n-4}>2$.     
Also, since $G$ has no chorded cycles, we have $e(G)=|A|+e(A,B)+e(B)=\gamma(u^*)+e(B)\leq 2n-4$. Combining this with \eqref{eq2}, we can deduce that $\gamma(u^*)=2n-4$ and $e(B)=0$.  Hence, $G[A\cup B]$ is a bipartite graph with coloring classes $A$ and $B$. Moreover, since $G$ does not contain a chorded cycle,  we assert that $G[A\cup B]$ does not contain a copy of $P_5$ with both endpoints in $A$. Thus, by Lemma \ref{lem::7}, 
\begin{equation}\label{eq5}
e(G[A\cup B])\leq |A|+2|B|-2.
\end{equation} 
On ther other hand, since  $\gamma(u^*)=|A|+e(A,B)=2n-4$, we have $e(G[A\cup B])=e(A,B)=2n-4-|A|=2(1+|A|+|B|)-4-|A|=|A|+2|B|-2$, that is,  the equality in \eqref{eq5} holds.  Again by Lemma \ref{lem::7}, we conclude that 
   $G[A\cup B] \cong K_{|A|,1}$ because $|A|>2$ and $|B|\geq 1$. Therefore, $|A|=n-1$, $|B|=1$, and $G\cong K_{2,n-2}$.

    {\flushleft \textit{Case 2.}} $e(A)\neq 0$.   
    
    Since $G$ contains no chorded cycles, we see that $G[A]$ is $P_3$-free. Let $A_0$ and $A_1$ denote the set of vertices with degree $0$ and $1$ in $G[A]$, respectively. Then  $A=A_0 \cup A_1$. For any $v\in A$, let $N'(v)=N(v)\cap B$.  As $G$ does not contain a chorded cycle, we have the following claim.

    \begin{claim}\label{claim::3}
     If $v_1\in A_0$ and $v_2\in A_1$ or $v_1,v_2\in A_1$, then  $N'(v_1)\cap N'(v_2)= \emptyset$.
     \end{claim}
     

Let  $N'(A_0)=\cup_{v\in A_0}N'(v)$ and $N'(A_1)=\cup_{v\in A_1}N'(v)$. Recall that a graph of order $n$ without a chorded cycle has at most $2n-4$ edges. By counting the number of edges in  $G[\{u^*\}\cup A_0\cup N'(A_0)]$, we obtain $e(\{u^*\}\cup A_0\cup N'(A_0))= e(A_0,N'(A_0))+e(\{u^*\},A_0)+e(N'(A_0))= e(A_0,N'(A_0))+|A_0|+e(N'(A_0))\leq 2(1+|A_0|+|N'(A_0)|)-4$, and hence
 \begin{equation}\label{eq3}
   e(A_0,N'(A_0))+e(N'(A_0))\leq |A_0|+2|N'(A_0)|-2.
   \end{equation}
According to \eqref{eq3} and Claim \ref{claim::3}, we get 
$$
e(A,B)=e(A_0,N'(A_0))+e(A_1,N'(A_1))\leq |A_0|+2|N'(A_0)|+|N'(A_1)|-2.
$$
Then it follows from (\ref{eq2}) that 
\begin{equation}\label{eq4}
\begin{aligned}
    2n-4&\leq \gamma(u^*)\\
    &=|A|+2e(A)+e(A,B)\\
    &\leq|A|+|A_1|+|A_0|+2|N'(A_0)|+|N'(A_1)|-2\\
    &=2|A|+2|N'(A_0)|+|N'(A_1)|-2,
    \end{aligned}
    \end{equation}
    which gives that 
    $$|A|+|N'(A_0)|+ \frac{|N'(A_1)|}{2}+1\geq n.$$
    On the other hand,  $|A|+|N'(A_0)|+|N'(A_1)|+1\leq n$. Therefore,  $N'(A_1)=\emptyset$ and $B=N'(A_0)$.  Combining this with \eqref{eq4}, we see that 
    \begin{equation}
    \begin{aligned}
      2n-4&\leq |A|+2e(A)+e(A,B)\\
    &\leq|A|+|A_1|+|A_0|+2|N'(A_0)|-2\\
    &=2|A|+2|N'(A_0)|-2\\
    &=2|A|+2|B|-2\\
    &= 2n-4.
    \nonumber
    \end{aligned}
    \end{equation}
Hence,  
\begin{equation}\label{eq6}
e(A,B)=e(A_0,N'(A_0))= |A_0|+2|N'(A_0)|-2.
\end{equation} 
Combining this with  \eqref{eq3}, we  deduce that $e(B)=e(N'(A_0))=0$. Furthermore, we claim that $B\neq \emptyset$. By contradiction, suppose that $B=\emptyset$. 
    If $n$ is odd, then $G$ is a spanning subgraph of $F_{\frac{n-1}{2}}$, and by Lemmas \ref{lem::2} and \ref{lem::3},  
    $$\rho(G)\leq \rho(F_{\frac{n-1}{2}})=\frac{1+\sqrt{1+4(n-1)}}{2}<\sqrt{2n-4}$$
    as $n\geq 6$, contrary to \eqref{eq1}. If $n$ is even, then $G$ is a spanning subgraph of $F^{u^*}_{\frac{n-2}{2}}$, and by Lemmas \ref{lem::2} and \ref{lem::5}, 
    $$\rho(G)\leq \rho(F^{u^*}_{\frac{n-2}{2}})< \sqrt{2n-4},$$ 
    again contrary to \eqref{eq1}. Therefore, $B\neq \emptyset$ and $e(B)=0$.
    
In what follows, we shall discuss according to the value of $|A_0|$. If $|A_0|=0$, then $B=N'(A_0)=\emptyset$, a contradiction.  If $|A_0|=1$, we suppose $A_0=\{v\}$. Let $G_2=G-\{vw: w\in B\}+\{u^*w: w\in B\}$. Clearly, $G_2$ contains no chorded cycles. By Lemma \ref{lem::1}, we have $\rho(G_2)> \rho(G)$, a contradiction. If $|A_0|\geq 2$, since $B=N'(A_0)\neq \emptyset$ and $e(B)=0$, we see that $G[A_0\cup B]$ is a bipartite graph with coloring classes $A_0$ and $B$. Since $G$ does not contain a chorded cycle,  we assert that $G[A_0\cup B]$ does not contain a copy of $P_5$ with both endpoints in $A_0$. Thus,  by \eqref{eq6} and Lemma \ref{lem::7},  we conclude that  $G[A_0\cup B]\cong K_{2,|B|}$ (in the case that $|A_0|=2$) or $G[A_0\cup B]\cong K_{|A_0|,1}$ (in the case that $|B|=1$).  If $G[A_0\cup B]\cong K_{2,|B|}$, then  $G\cong K_{2,|B|+1}\bullet
     F_{\frac{n-|B|-3}{2}}$, and by Lemma \ref{lem::6},
    $$\rho(G)= \rho(K_{2,|B|+1}\bullet F_{\frac{n-|B|-3}{2}})<\sqrt{2n-4},$$
     contrary to \eqref{eq1}. If $G[A_0\cup B]\cong K_{|A_0|,1}$,  then  $G\cong K_{2,|A_0|}*  F_{\frac{n-2-|A_0|}{2}}$, and again by Lemma \ref{lem::6}, 
   $$\rho(G)= \rho(K_{2,|A_0|}* F_{\frac{n-2-|A_0|}{2}})<\sqrt{2n-4},$$ 
   which is impossible.

 Therefore, we conclude that $G\cong K_{2,n-2}$, and the result follows.\qed

\section*{Acknowledgement}
X. Huang is supported by the National Natural Science Foundation of China (Grant No. 11901540).

\section*{Data Availability}
 Data sharing not applicable to this article as no datasets were generated or analysed during the current study.

\section*{Declarations}
\textbf{Conflict of interest}  The authors declare no competing financial interests.


\begin{thebibliography}{99}
		\setlength{\itemsep}{0pt}
		
		
		\bibitem{B} R. B. Bapat, Graphs and Matrices, Springer, New York, 2010.
		
		

		
		
		 \bibitem{BZ} A. Berman, X.-D. Zhang, On the spectral radius of graphs with cut vertices, J. Combin. Theory Ser. B 83(2) (2001) 233--240.
		 
		 \bibitem{BH} A.E. Brouwer, W.H. Haemers, Spectra of Graphs, Springer, Berlin,  2011.
		
		
		\bibitem{CFTZ} S. M. Cioab\u{a}, L. Feng, M. Tait, X.-D. Zhang, The maximum spectral radius of graphs without friendship subgraphs, Electron. J. Combin. 27(4) (2020) \#P4.22.
		
		
		\bibitem{CGH} S. M. Cioab\u{a}, D. A. Gregory, W. H. Haemers,  Matchings in regular graphs from eigenvalues, J. Combin. Theory Ser. B 99 (2009) 287--297.
		
		\bibitem{FN} M. Fiedler,  V. Nikiforov, Spectral radius and Hamiltonicity of graphs, Linear Algebra Appl. 432(9) (2010) 2170--2173.
		
		\bibitem{GR} C. Godsil, G. Royle, Algebraic Graph Theory, Graduate Texts in Mathematics, 207, Springer-Verlag, New York, 2001.
		
		\bibitem{G}  R. J. Gould, Results and problems on chorded cycles: A survey,  Graphs Combin. 38 (2022)  189.
		
		
		
		\bibitem{LN} B. Li, B. Ning, Spectral analogues of Erd\H{o}s' and Moon-Moser’s theorems on Hamilton cycles, Linear Multilinear Algebra 64(11) (2016) 2252–2269.
		
		\bibitem{O16} S. O, Spectral radius and fractional matchings in graphs, European J. Combin. 55 (2016) 144--148. 
		
		\bibitem{O22} S. O, Eigenvalues and $[a,b]$-factors in regular graphs,  J. Graph Theory 100(3) (2021) 458--469.
		
		\bibitem{OC} S. O, S. M. Cioab\u{a}, Edge-connectivity, eigenvalues, and matchings in regular graphs, SIAM J. Discrete Math. 24(4) (2010) 1470--1481.
		
		
		\bibitem{OPPZ} S. O, J. R. Park, J. Park, W. Zhang, Sharp spectral bounds for the edge-connectivity of regular graphs, European J. Combin. 110 (2023) 103713.
		
		\bibitem{P}  L. P\'{o}sa, Problem No. 127 (Hungarian), Mat. Lapok 12 (1961) 254.
		
		\bibitem{V}  J. R. Vermette, Spectral and combinatorial properties of friendship graphs, simplicial rook graphs, and extremal expanders. Ph.D. thesis, University of Delaware, 2015.
		
		\bibitem{WKX} J. Wang, L. Kang, Y. Xue, On a conjecture of spectral extremal problems, J. Combin. Theory Ser. B 159 (2023) 20--41.
		
		\bibitem{XLLS}  J. Xue, H. Lin, S. Liu, J. Shu, On the $A_{\alpha}$-spectral radius of a graph, Linear Algebra Appl. 550 (2018)  105--120.
		
		
	    \bibitem{ZL} M. Zhai, H. Lin, Spectral extrema of $K_{s,t}$-minor free graphs --- on a conjecture of M. Tait, J. Combin. Theory Ser. B 157 (2022) 184--215.
	    
	    \bibitem {ZLG15}  M. Zhai, H. Lin, S. Gong, Spectral conditions for the existence of specified paths and cycles in graphs, Linear Algebra Appl. 471 (2015) 21--27.		
			
	\end{thebibliography}
\end{document}